\documentclass[reqno]{amsart}
\usepackage{amsmath,amsthm}
\usepackage{amssymb}
\usepackage{xypic}
\usepackage{mathrsfs}
\usepackage{hyperref}
\usepackage{graphicx}
\usepackage{perpage}
\usepackage{extarrows}
\MakePerPage{footnote}
\usepackage{setspace}
\date{\today}
          
 \usepackage[usenames,dvipsnames]{pstricks}
\usepackage[shell]{pdftricks}
 \begin{psinputs}
 \usepackage{pst-grad} 
 \usepackage{pst-plot} 
 \usepackage{epsfig}
 \end{psinputs}

\title{Cohomology of flat currents on definable pseudomanifolds}

\author{Saurabh Trivedi}

\address{ICMC-USP, S\~ao Carlos, SP, Brazil}

\email{saurabh.trivedi@gmail.com}

\begin{document}
\newtheorem{thm}{Theorem}[section]
\newtheorem{lem}[thm]{Lemma}

\theoremstyle{theorem}
\newtheorem{prop}[thm]{Proposition}
\newtheorem{defn}[thm]{Definition}
\newtheorem{cor}[thm]{Corollary}
\newtheorem{example}[thm]{Example}
\newtheorem{xca}[thm]{Exercise}

\newcommand{\bb}{\mathbb}
\newcommand{\al}{\mathcal}
\newcommand{\ak}{\mathfrak}
\newcommand{\fs}{\mathscr}

\theoremstyle{remark}
\newtheorem{rem}[thm]{Remark}
\parskip .12cm

\begin{abstract} We show that the cohomology of flat currents on definable pseudomanifolds in polynomially bounded o-minimal structures is isomorphic to its intersection cohomology in the top perversity. 
\end{abstract}

\maketitle

It is known that the intersection cohomology of pseudomanifolds is related to their $L^p$-cohomology. Cheeger \cite{Cheeger} showed that the $L^2$-cohomology of pseudomanifolds with metrically conical singularities is isomorphic to their intersection homology in the middle perversity. Youssin \cite{Youssin} proved that the $L^p$-cohomology ($p \neq 2$) of spaces with conical horns is isomorphic to the intersection homology in what is called the $L^p$-perversity. Saper \cite{Saper} proved a similar result for cohomology of spaces with isolated singularities with a distinguished K\"ahler metric. And, Hsiang and Pati \cite{Pati} showed that the $L^2$-cohomology of normal algebraic complex surfaces is dual to their intersection homology. 

In all the above mentioned articles the results are proved for spaces either in low dimensions or with assumptions on their singularities. In this article we consider flat currents on pseudomanifolds in polynomially bounded o-minimal structures and show that the cohomology of flat currents on such sets is isomorphic to their intersection cohomology in the top perversity. We do not make any assumptions on the dimensions or the metric type of singularities of the pseudomanifolds under consideration. 

In section 1, we give definitions of o-minimal structures, polynomially bounded o-minimal structures, definable pseudomanifolds and currents on them.

In section 2, we define flat currents on definable pseudomanifolds. A classical result about the structure of flat currents is then mentioned. Flat currents have a well defined boundary operator and this allows us to define a cohomology theory for pseudomanifolds, which we call the flat cohomology.

Our aim is to prove a Poincar\'e lemma for flat currents. For this we define the notion of weak flat currents in section 3. A result relating the flat cohomology and cohomology of weak flat currents, that they are isomorphic, is then mentioned. 

In section 4, we prove a local Poincar\'e lemma for flat currents on definable pseudomanifolds in polynomially bounded o-minimal structures. A key result used to prove the lemma is the local Lipschitz retractibility of definable sets in polynomially bounded o-minimal structures.

In section 5, we recall the definition of normal pseudomanifolds and existence of normalizations of pseudomanifolds. We then state a result due to Goresky and Macpherson which says that the intersection cohomology in top perversity is isomorphic to the singular cohomology of the normalization of a pseudomanifold. A definable version of this results is then proved.

Finally, in section 6, a global version of Poincar\'e lemma is proved. This gives us the desired de Rham theorem, relating the flat cohomology and the intersection cohomology in top perversity of definable pseudomanifolds.

\section{Definitions}

\subsection{o-minimal structures}
 A family $\al D = \{\al D_n\}_{n\in \bb N}$, where $\al D_n$ is a collection of subsets (called definable sets of $\al D$) of $\bb R^n$, is called an o-minimal structure on $\bb R$ if:
 
 1. $\al D_n$ is closed under union, intersection and complements.\\
 \indent
 2. if $A \in \al D_n$ then $\bb R \times A$ and $A \times \bb R$ are in $\al D_{n+1}$.\\
 \indent
 3. $\al D_n$ contains the zero set of all polynomials in $n$-variables.\\
 \indent
 4. if $A \in \al D_n$ then its projection onto $\bb R^{n-1}$ is in $\al D_{n-1}$.\\
 \indent
 5. The members of $\al D_1$ are finite unions of points and intervals.
 
 A map $f : A \to B$ between two definable sets is called a definable map if its graph is definable.
 
 The family of semialgebraic sets, globally subanalytic sets and extension of globally subanalytic sets by log and exponential sets (called the log-exp structure) are examples of o-minimal structures.
 
 An o-minimal structure is said to be polynomially bounded if for each definable function $f : \bb R \to \bb R$, there exists an $a>0$ an integer $r > 0$ such that $f(x) \leq x^r$ for all $x>a$.
 
 Semialgebraic sets and globally subanalytic sets are examples of polynomially bounded o-minimal structures on $\bb R$. Log-exp structure is an example of a non-polynomially bounded o-minimal structure.
 
 Definable sets in o-minimal structures can be decomposed into cells; see theorem 2.11 in \cite{Dries1}. By the dimension of a definable set we mean the maximum of the dimensions of cells in its cell decomposition. We must clarify here that there exist o-minimal structures whose definable sets do not admit a smooth cell decomposition. Examples can be found in le Gal and Rolin \cite{leGal}. However, every o-minimal structure admits a $C^r$-cell decomposition ($r$ depending on the structure). The set of semialgebraic sets, or subanalytic sets admit smooth cell decomposition.  
 
 In this article we work only with \emph{polynomially bounded o-minimal structures that admit smooth cell decomposition}. The reason for this is two fold: first is that we implicitly use a preparation theorem for definable maps in polynomially bounded o-minimal structures\footnote{There is no preparation theorem for o-minimal structures in general.} and second that we consider smooth differential forms on definable manifolds whose existence depends on the degree of smoothness of the cell decomposition. For a precise statement of the preparation theorem for definable maps in polynomially bounded structures we refer to Nguyen and Valette \cite{ValetteNguyen}.
 
 \subsection{Definable pseudomanifolds} Let $\al D = \{\al D_n\}_{n\in \bb N}$ be an o-minimal structure on $\bb R$. Let $X \subset \bb R^n$ be a definable set in $\al D$.  We denote by $M \subset X$ the set of points of $X$ at which $X$ is locally a smooth submanifold of $\bb R^n$, it is called the regular locus of $X$. The singular locus of $X$ is $X - M$. Both $M$ and $X - M$ are again definable sets in $\al D$. In general $M$ can have multiple components with each component being a manifold of different dimension. Thus, we assume in addition that $M$ is dense in $X$ and in that case it will be a $C^r$-submanifold of $\bb R^n$ of dimension $l$.
 
 A definable set $X$ in $\al D$ of dimension $l$ is said to be a pseudomanifold if the dimension of its singular locus $X - M$ is at most $l-2$. Typical examples of pseudomanifolds are pinched tori, suspension of tori and the wedge of spheres. In fact, any complex projective variety is a definable pseudomanifold in the o-minimal structure of globally subanalytic sets. 
 
 \subsection{Currents} Let $M$ be a smooth manifold of dimension $l$. Denote by $\Omega^q_c(M)$ the set of differential $q$-forms on $M$ with compact support. Put the topology on $\Omega^q_c(M)$ characterized by the assertion that a sequence $\{\omega_i\}$ in $\Omega^q_c(M)$ converges to a compactly supported $q$-form $\omega$, if there exists a compact set $K \subset M$ such that all $\omega_i$'s have support in $K$ and the derivatives of the coefficients of $\omega_i$'s converge uniformly to the derivatives of coefficients of $\omega$. Let us use $\Omega^q_c(M)$ for the set of compactly supported differential $q$-forms with this topology. 
 
 A $p$-current on $M$ is a continuous linear functional on $\Omega_c^{l-p}(M)$. Notice that we define $p$-currents to be continuous linear functional on differential $(l-p)$-forms. Just why is it so will be clear later. Notice also that $l$-currents on $M$ are nothing but distributions on $M$. 
 
 We denote by $\fs C^p(M)$ the set of all $p$-currents on $M$. There is a canonical operator on currents called the boundary operator. It is the linear map $\partial:\fs C^{p}(M) \rightarrow \fs C^{p+1}(M)$ defined by $(\partial T)(\omega) = T(d\omega)$, where $d$ is the exterior derivative on forms. Notice that $\partial \partial T =0$ for any current $T$. Just like in the case of forms, a $p$-current $T$ is said to be closed if $\partial T = 0$, that is $\partial T(\omega) = 0$ for all $\omega$ and $T$ is said to be exact if there exists a $(p-1)$-current $S$ such that $\partial S = T$. 
 
 \section{Flat currents on Definable Pseudomanifolds} 
 
 Let $X \subset \bb R^n$ be a definable pseudomanifold of dimension $l$ with regular locus $M$ in an o-minimal structure $\al D$ on $\bb R$. Let $\Omega^q_c(X)$ be the set of differential $q$-forms on $M$ ($M$ is a smooth submanifold of $\bb R^n$) with compact support in $\bb R^n$. Since $X$ is dense in $M$ the closure of $M$ is $X$ thus it is same as saying that the support lies in $X$ and is compact. 
 
 Denote by $\fs C^p(X)$ the set of all $p$-currents on $X$; that is continuous linear functionals on $\Omega^{l-p}_c(X)$ (with the topology as defined before). 
 
 \subsection{Mass and Flat norm}
 For a current $T \in \fs C^p(X)$, define its mass by:
 $$M(T) = \sup\{T(\omega) : \omega \in \Omega^{l-p}(X), |\omega|<1 \}$$
 where
 $$|\omega| = \sup\{|\langle \omega,\xi\rangle| : \xi\,\, \text{is unit, simple, $(l-p)$-vector}\}.$$ Notice that the mass is a seminorm.
  
 We define another seminorm on currents called the flat norm due to Whitney. The flat norm of $T$ is defined to be:
 $$\ak F(T) = \inf \{M(T-\partial A) + M(A) : A \in \fs C^{p-1}(X)\}.$$
 
 \subsection{Normal and Flat currents}
 
 A $p$-current $T$ is said to be normal if $M(T) + M(\partial T)$ is finite.
 
 The elements of closure of normal currents under the flat norm are called flat currents. We denote by $\al F^p(X)$ The set of flat $p$-currents on $X$.
 
 We have a result of Federer classifying the flat currents as follows:
 
 \begin{thm}\label{thm21}
 A $p$-current $T$ is flat if and only if there exist a flat $p$-current $R$ and a flat $(p-1)$-current $S$ both of finite mass such that $T = R + \partial S$.
 \end{thm}

It follows immediately from the above theorem that the boundary of a flat current is again flat. 
 
 \subsection{Flat cohomology} Since the boundary of a flat current is also a flat current. It follows that the flat currents along with the boundary operator $\partial$ form a co-chain complex. The cohomology of the co-chain complex of flat currents is called the flat cohomology. We denote it by $\al FH^*(X)$. And, this is why we defined $p$-currents as the linear functional on $(l-p)$-forms. 
 
 \section{Weak Flat currents}
 
 In this section we define the notion of weak flat currents on definable pseudomanifolds and the cohomology induced by them. This will be useful in proving a local Poincar\'e lemma for flat currents. For this we need first the notion of weak differential forms. The definition is lifted from Valette \cite{Valette1}. Although the definition can be given in a more general setting, we restrict ourselves to currents on definable pseudomanifolds.
 
Let $X$ be a definable pseudomanifold of dimension $l$ with regular part $M$. A continuous differential $p$-form $\alpha$ on $M$ with compact support in $X$ is said to be weakly differentiable if there exist a continuous differential $(p+1)$-form $\omega$ on $M$ with compact support in $X$ such that for every $(l-(p+1))$-form $\theta \in \Omega^{l-(p+1)}_c(X)$ we have:
$$\int_M \alpha \wedge d \theta = (-1)^{(p+1)}\int_M \omega \wedge \theta.$$
And, we define the weak derivative of $\alpha$ as $\overline{d}\alpha = \omega$.

Denote by $\overline{\Omega^p_c(X)}$ the set of all weakly differentiable forms on $M$ with compact support in $X$. Since every smooth form is also weakly differentiable, we have $\Omega^p_c(X) \subset \overline{\Omega^p_c(X)}$. In fact, with respect to the $L^{\infty}$-norm on forms defined in 2.1 this inclusion is dense. This can be seen as follows: Let $\alpha$ be a weakly differential form with weak derivative $\omega$ and $\eta_{\epsilon}$ be the standard bump function on $X$, then the convolution $\{\omega \ast \eta_{\epsilon}\}$, which a smooth approximation of $\alpha$ as $\epsilon$ tends to $0$.

Now, if $T$ is a flat current on $X$ and $\{\omega_n\}$ is a sequence of smooth forms on $M$ with compact support in $X$ with finite $L^\infty$-norm that converges to a weakly differentiable form $\tilde{\omega}$, we can define a linear functional $\overline T$ on the set of weakly differentiable forms by setting $$\overline{T}(\tilde{\omega}) = \lim_{n \to \infty} T(\omega_n).$$ We call this linear operator a weak flat current. One can analogously talk about weak normal currents, weak mass norm and weak flat norm etc. This will be used later.

Denote by $\overline{\al F^p(X)}$ the set of all weak flat currents. A boundary operator $\overline \partial$ on weak flat currents can be defined by setting 
$$\overline{\partial}\,\overline{T}(\tilde{\omega}) = \overline{T}(\overline d\, \tilde{\omega}).$$

If $\overline{T}$ is a weak flat current, then
\begin{align*}
\overline{\partial}\,\overline{T}(\tilde{\omega})=\overline{T}(\overline{d}\tilde{\omega}) = \lim_{n \to \infty} T(d \omega_n) = \lim_{n \to \infty} \partial T(\omega_n) = \overline{\partial T}(\tilde{\omega})
\end{align*}
This implies
$$\overline{\partial}\,\overline{\partial}\,\overline{T}(\tilde{\omega}) = \overline{\partial\partial T}(\tilde{\omega}) = 0.$$
Thus, the set of weak flat currents $\overline {\al F^p(X)}$ with the boundary operator $\overline \partial$ form a cochain complex whose cohomology is called the weak flat cohomology. We denote this cohomology by $\overline{\al F}H^*(X)$.
 
 Notice that $T \to \overline T$ gives a canonical chain maps from the set of flat currents to the set of weak flat currents on $X$. This induces a map on the cohomology $\al FH^*(X) \to \overline{\al F}H^*(X)$. 
 
 A weak flat current is defined as the unique extension of a flat current to a linear functional on weakly differentiable forms. Since any flat current extends, this implies that weak flat currents are in bijection with flat currents. The boundary operator on weak flat currents commutes with this extension, so the cohomology groups of flat and weak flat currents are isomorphic. We have the following lemma:
 
 \begin{lem}\label{Inc} The mapping $\phi:\al F H^p(X) \rightarrow \overline{\al F}H^p(X)$ induced by extending flat currents to weakly weakly flat currents is an isomorphism. Moreover, the restriction of weak flat currents on smooth forms induces the inverse of $\phi$ on the cohomology level.
 \end{lem} 
 
\section{A local Poincar\'e lemma for flat currents}

We prove a local Poincar\'e lemma for flat currents here. This will be used in proving a global results about the flat cohomology. 

In the following $X \subset \bb R^n$ is a locally closed definable pseudomanifold in a polynomially bounded o-minimal structure that admits smooth cell decomposition of dimension $l$ with regular part $M$.  Let $U \subset X$ be an open set in $X$ and denote by $\Omega^p_c(\overline{U})$ the set of smooth $p$-forms on $U \cap M$ that extend to a neighbourhood $V$ of $\overline{U}$ and have support in $\overline U$.

We first state the local Lipschitz retractibility result. Although this result is proved for subanalytic sets in Valette \cite{Valette1}, the proof works also for any definable set in a polynomially bounded o-minimal structure; see Remark 4.3.4 in \cite{Valette1}.

\begin{lem} \label{Lipret}
 For every $x_0 \in X$ there exist an open ball $U = B^n(x_0,\epsilon)$ of radius $\epsilon$ and a Lipschitz retraction $r : (U \cap X) \times [0,1] \rightarrow (U \cap X)$ onto $x_0$, i.e. for each $t \in [0,1]$, $r_t : (U \cap X)  \rightarrow (U\cap X)$ defined by $r_t(x) = r(x,t)$ is a Lipschitz map such that $r_0(x) = x_0$ and $r_1$ is identity on $X$. Moreover, $r_t$ preserves $U \cap M$ for $t>0$ and the derivative $d_xr_t$ tends to $0$ as $t$ tends to $0$.
\end{lem}

The idea of the proof of our result is similar to the classical Poincar\'e lemma except there are two obstructions that we encounter. The first obstruction is that we need to take pull back of smooth forms under a Lipschitz retraction which might not be smooth. The second obstruction as we will see is that the integral involved in the proof could be an improper integral. We overcome the first obstruction by using an approximation of our Lipschitz retraction and the second by showing that the integral gives a weakly differential form. We need the following lemmas:

\begin{lem}{\label{lem_homop}}
For $0 < p \leq l$, there exists a linear map 
\begin{equation*}
\al K : \Omega^p_c(B^n(\epsilon;x_0) \cap X) \rightarrow \overline{\Omega}^{p-1}_c(B^n(\epsilon;x_0) \cap X) 
\end{equation*} such that $\al K d +\overline{d} \al K$ is the identity on $\Omega^p_c(B^n(\epsilon;x_0) \cap X)$.
\end{lem}

\begin{proof}
By Lemma \ref{Lipret} there exists a Lipschitz retraction $r : (B^n(\epsilon;x_0) \cap X) \times [0,1] \rightarrow B^n(\epsilon;x_0) \cap X$. Let $r': (B^n(\epsilon;x_0) \cap M) \times (0,1) \rightarrow B^n(\epsilon;x_0) \cap M)$ be a smooth approximation of $r$ in the sense that, given any positive continuous function $\xi : (B^n(\epsilon;x_0) \cap M) \times (0,1) \rightarrow \bb R$ decreasing fast enough there exists $r'$ such that $|r'_t(x) - r_t(x)| \leq \xi(x,t)$; see Lemma 4.2.1 in \cite{Valette1} for the construction of such a map. This $r'$ has the following properties:
	
1. The first partial derivatives of $r'$ are bounded above.
	
2. $\lim_{t \to 0} d_xr'_t = 0$ and $\lim_{t\to 1} d_xr'_t$ tends to the identity map for almost every $x$.
	
3. $r'$ preserves $M$.
	
Now, If $p > 0$, for any $p$-form $\omega$ on $B^n(\epsilon;x_0) \cap M$  we can uniquely write the $p$-form $r'^*\omega$ on $(B^n(\epsilon;x_0) \cap M) \times \bb (0,1)$ as
$$r'^*\omega = \omega_1 + \omega_2 \wedge dt,$$
where $dt$ is the standard 1-form on $(0,1)$, $\omega_1$ is a $p$-form on $(B^n(\epsilon;x_0) \cap M)\times (0,1)$ that is free from $dt$ and $\omega_2$ is a $(p-1)$-form on $B^n(\epsilon;x_0) \cap M$. 
	
Put $$\al K(\omega) = \int_{0}^1 \omega_2(x,t) dt.$$ Since the support of $\omega_2$ lies in $\overline{B^n(\epsilon;x_0) \cap M}$, it is an improper integral. We show that $\al K(\omega)$ is weakly differentiable.
	
Recall that if $\phi$ is a differential $(l-p)$-form of class $C^2$ with support in $M$ and $\alpha$ any $p$-form, then the derivative of the exterior product is given by 
$$d(\alpha \wedge \phi)  = d \alpha \wedge \phi + (-1)^p \alpha \wedge d\phi.$$ Thus, for $\alpha = r'^*(\omega)$ we have:
\begin{align*}\label{dagger}
r'^*(\omega) \wedge = (-1)^p\{d(r'^*(\omega)\wedge \phi) - d(r'^*(\omega)\wedge \phi)\}. \tag{$\dagger$}
\end{align*}	
Then,
\begin{align*}
\int_M \al K(\omega) \wedge d\phi  =& \int_M \int_0^1 \omega_2 \wedge d\phi =\quad \lim_{t\to 0} \int_{M \times [t;1]} r'^*\omega \wedge d\phi\\
=&\quad \lim_{t\to 0}(-1)^p \left \{ \int_{M \times [t;1]}  d(r'^*\omega \wedge \phi) - \int_{M \times [t,1]} d(r'^*(\omega)) \wedge \phi \right \}\tag{by \ref{dagger}}\\
=& \quad\lim_{t\to 0}(-1)^p \bigg \{ - \int_{M \times \{t\}}  \omega_1 \wedge \phi + \int_{M \times \{1\}} r'^*\omega \wedge \phi \\
&\quad - \int_{M \times [t,1]} d (\omega_1 + \omega_2 dt) \wedge \phi \bigg \} \tag{By Stokes' formula}\\
=&\quad  (-1)^p\left \{ \int_{M} \omega \wedge \phi - \int_M \int_0^1 d\omega_2 \wedge \phi\right \} \tag{Since $\lim_{t \to 0} \omega_1(x,t) = 0$, $\lim_{t\to 1} r'^*\omega(x,t) = \omega(x)$}
\end{align*}
Finally, since the pullback commutes with the exterior derivative, we have:
\begin{align*}
\int_M \al K(\omega) \wedge d\phi	=&\quad (-1)^p\int_M (\omega - \al Kd\omega) \wedge \phi
\end{align*}
	
This implies that $\al K \omega$ is weakly differentiable and $\overline d\al K \omega=\omega - \al K d \omega$. This also shows that if $\al K$ is considered as a linear map from $\Omega^p_c(B^n(\epsilon;x_0))$ to $\overline{\Omega}^{p-1}_c(B^n(\epsilon;x_0) \cap M)$, then $\al K d + \overline d \al K$ is the identity on $\Omega^p_c(B^n(\epsilon;x_0) \cap M)$. This concludes the Lemma.
\end{proof}

We now show that the linear operator $K$ obtained from the above lemma induces a linear operator on currents. This operator on currents then gives the required Poincar\'e lemma.

\begin{lem}\label{lem_fhomop} Let $p \geq 0$ and let $T$ be a closed weak flat $(p+1)$-current on $B^n(\epsilon;x_0) \cap X$, then there exists a weak flat ${p}$-current $S$ on $B^n(\epsilon;x_0) \cap M$ such that $\overline{\partial} S = T$.
\end{lem} 

\begin{proof} 
	
	
	 The operator $\al K : \Omega^{l-p}_c(B^n(\epsilon;x_0) \cap M) \rightarrow \overline{\Omega}^{l-p-1}_c(B^n(\epsilon;x_0) \cap M)$ provided by Lemma \ref{lem_homop} induces a linear operator 
	\begin{equation*}
	\al K^{\#} : \overline{\fs C^{p+1}}(B^n(\epsilon;x_0) \cap M) \rightarrow  \overline{\fs C}^{p}(B^n(\epsilon;x_0) \cap M)
	\end{equation*}
	defined by $\al K^{\#}(T)(\omega) = T(\al K(\omega))$; here $\overline{\fs C}$ stands for the set of weak currents. Since $\al K d + \overline d \al K$ is the identity, it is easy to see that
	\begin{equation}
	\label{Chomop}
	\overline{\partial} \al K^{\#}(T) + \al K^{\#} \overline{\partial}(T) = T.
	\end{equation}
	Thus if $T$ is a weak closed current, by (\ref{Chomop}), if we put $\al K^{\#}(T)=S$, we have $\overline{\partial} S = T$.
	
	It remains to show that if $T$ is a weak flat current then so is $S$. So suppose $T$ is a weak flat current. Then, by definition there exists a sequence of weak normal currents $T_n$ that converges to $T$ under the flat norm. We will show that for large $n$, $\al K^{\#} T_n$ are also weak normal currents and that $\al K^{\#} T_n$ converge to $\al K^{\#} T$ under the flat norm.
	
	To prove that $\al K^{\#}$ maps normal currents to normal currents it is enough to show that if $T$ has finite mass then so has $\al K^{\#} T$. First notice that the Lipschitz retration $r'$, as in the previous lemma, has bounded partial derivatives. So, if $\omega$ is a bounded differential form then so is $K(\omega)$, and moreover $|K(\omega)| < |\omega|$. Now, if $T$ is such that $M(T) < \infty$. Then,
	
	\begin{align*}
	M (\al K^{\#} T) &= M (T(K)) \\
	&= \sup \{T(K(\omega)) : |\omega| <1\} \\
	& \leq \sup \{T(\omega)  : |\omega| < 1\} \tag{since $|K(\omega)| < |\omega|$}  \\
	& \leq M(T) < \infty \tag{since $T$ is normal}
	\end{align*}
	
	Finally we show that $K^{\#}T_n$ converge to $K^{\#}T$ under the flat norm. Since $T_n$ converges to $T$ under the flat norm, for every $\epsilon > 0$ there exists an $n_0 \in \bb N$ such that for all $n > n_0$, 
	$\ak F (T_n - T) < \epsilon$. But then,
	
	\begin{align*}
	\ak F(\al K^{\#} T_n - \al K^{\#}T)  &= \ak F({T_n(K) - T(K)})\\
	&= \ak F((T_n - T)(K))\\
	&= \inf \{M((T_n - T)(K) - \partial A) + M(A)\}\\
	&\leq \inf \{M((T_n - T) - \partial A) + M(A)\} \\
	&\leq \ak F(T_n - T) < \epsilon
	\end{align*}
	
	Thus this shows that $\al K^{\#}$ maps weak flat currents to weak flat currents. This concludes the lemma.
\end{proof}
 From Lemma \ref{Inc}, Lemma \ref{lem_homop} and Lemma \ref{lem_fhomop} we immediately get the local Poincar\'e lemma for flat currents on definable pseudomanifolds in polynomially bounded o-minimal structures:
 
 \begin{thm}\label{lem_Poincare}
 	Let $X \subset \bb R^n$ be a definable pseudomanifold in a polynomially bounded o-minimal structure of dimension $l$ and let $M$ be its regular part. Then, for every $x \in X$ there exists a Lipschitz retractible neighbourhood $U$ of $X$ around $x$ such that any closed flat $p$-current ($0< p \leq l$), on $U \cap M$ is exact.
 \end{thm}
 
 \section{Normal pseudomanifolds}
 
 Let $X \subset \bb R^n$ be a definable pseudomanifold in a polynomially bounded o-minimal structure and $M$ be its regular part. The link $L(X,x)$ of $X$ at a point $x \in X$ is the intersection $S^{n-1}(\epsilon; x) \cap X$ of $X$ with a small sphere. Then, $X$ is said to be normal if $L(X,x)$ is connected at all points $x \in X$. We have the following result about the existence of normalizations of pseudomanifolds; see Sections 4.1 and 4.2 in Goresky and Macpherson \cite{Goresky1}:
 
 \begin{thm}\label{existnorm} For any pseudomanifold $X$, there exists a normal pseudomanifold $\tilde{X}$ (normalization of $X$) and a finite-to-one continuous map $\pi : \tilde{X} \to X$ such that $\pi^{-1}(M)$ is homeomorphic to $M$.
 \end{thm}
 
 The relation between the intersection cohomology\footnote{The original statement relates only homology but it is easy to see that it passes on to intersection cohomology as well.} of $X$ and the singular cohomology of $X$ is as follows.

 \begin{thm}\label{GMnorm} Let $X$ be a pseudomanifold. Then the intersection cohomology of $X$ in the top perversity is isomorphic to the singular cohomology of its normalization.
 \end{thm}

The above results do not guarantee that the normalizations are definable. We will show that every definable pseudomanifold admits a definable normalization. The following lemma applies to definable sets in any $o$-minimal structure. The construction is similar to that given in \cite{Goresky1}.

\begin{lem}\label{lemnorm}
For any definable pseudomanifold $X$ with regular part $M$, there exists normal definable pseudomanifold $\tilde{X}$ with a continuous map $\pi : \tilde{X} \to X$ such that $\pi^{-1}(M)$ is homeomorphic to $M$. 
\end{lem}

\begin{proof}
We know that any definable set admits a $C^0$-triangulation; see Coste \cite{Coste}. Suppose $T : K \to X$ is a triangulation of $X$, i.e. $T$ is a homeomorphism from a finite union of open simplices $K$ onto $X$. 

Let $L$ be the disjoint union of all the closures in $K$ of the $l$-dimensional open simplices of $K$ ($\dim X = l$). By identifying the closure in $K$ of two $(l-1)$-dimensional open faces of two elements of $L$ if these two faces coincide in $K$, we obtain a simplicial complex $\tilde{X}$. Since simplicial complexes are definable and identification and taking closures are definable conditions, $\tilde{X}$ is again a definable set. Denote by $\pi : \tilde{X} \to X$ the map induced by $T$. 

By construction it is clear that $\pi$ is a homeomorphism on the complement in $X$ of the $(l-2)$-skeleton. Thus, it is clear from the definition that the mapping $\pi$ induces a homeormophism over $M$ and that the link is connected at all singular points. This completes the lemma.
\end{proof}

\section{de Rham theorem for flat currents}

 Let $X \subset \bb R^n$ be a definable pseudomanifold in a polynomially bounded o-minimal structure and $M$ be its regular part. 
 
 \begin{lem}\label{zerohom} If $M$ is connected then the $0$-th flat cohomology group of $X$ $\al F H^0(X)$ is $\bb R$. 
\end{lem}
\begin{proof}
 If $T$ is a $0$-current then it is a linear functional on $l$-forms on $M$ with compact support in $X$, thus $\partial T(\omega) = T(d \omega) = 0$. This implies that all $0$-currents are closed. We show that $\al FH^0(X)$ is isomorphic to $\bb R$. The function on $\al F^0(X)$ defined by sending an element $T$ to $T(\omega)$, where $\int_M \omega = 1$ is a generator of the $l$-th compact de Rham cohomology group of $X$, is certainly surjective. We show that it is also injective. So, suppose $T$ is a current such that $T (\omega) = 0$, then since the $l$-th compact cohomology group of $X$ is $\bb R$, every $l$-form $\gamma$ on $X$ can be written as $\gamma = a \omega + d\alpha$. Then $T(\gamma) = a T(\omega) + T(d \alpha) = \partial T(\gamma) =0$. This implies that $T=0$ and $\al FH^0(X) \simeq \bb R$.
 \end{proof}
 
 The main result is as follows:
 
 \begin{thm} Let $\al D$ be a polynomially bounded o-minimal structure on $\bb R$ and $X \subset \bb R^n$ be a definable pseudomanifold in $\al D$. Then, the intersection cohomology in the top perversity of $X$ in $\al D$ is isomorphic to the flat cohomology of $X$. 
 \end{thm}
 
 \begin{proof}
 Let $M$ be the regular part of $X$. Denote by $\al F^j(X)$ the set of flat currents on $X$. Let $\pi : \tilde{X} \to X$ be a normalization of $X$, which exists by Lemma \ref{lemnorm}. Define a presheaf $\tilde{\al S}^j$ on $\tilde{X}$ by setting  $\tilde{\al S}^j (U) = \al F^j(\pi(U) \cap M)$ for any open set $U$. That $\tilde{\al S}^j$ is a presheaf can be seen as follows:
 
 Given open sets $V \subset U$ of $X$, and $T \in \al F^j(\pi(U)\cap M)$ we can define a restriction $T^U_V \in \al F^j(\pi(V) \cap M)$ by setting $T^U_V(\omega) = T(\overline{\omega})$, where $\overline{\omega}$ is an extension of $\omega$ on $\pi(U) \cap M$ by $0$. The extension is well defined because $\omega$ has compact support in $\pi(V) \cap M$ and thus in $\pi(U) \cap M$.


 By abuse of notation, we denote still by $\tilde{\al S}^j$ the sheafification of this presheaf. Since $\pi$ is a homemorphism above $M$, the global sections of $\tilde{\al S}^j$, denoted $\tilde{\al S}^j(X)$, will be flat currents on $X$. Thus, $\pi$ induces a homeomorphism of cochain complexes $\tilde{\al S}^*(\tilde{X}) \equiv \al F^*(X)$.
 
 Now, let $x_0 \in \tilde{X}$ and set $U=B^n(\epsilon,x_0) \cap \tilde{X}$ for a small enough $\epsilon$. Since $\tilde{X}$ is a normalization,  $\pi(U) \cap M$ is connected, thus the zero order flat cohomology of $\pi(U)\cap M$ is $\bb R$ by Lemma \ref{zerohom}. 
 
 By the Poincar\'e Lemma \ref{lem_Poincare}, the germ at $\pi(x_0)$ of a closed flat current on $\pi(U) \cap M$, is the boundary of an element of the stalk of $\tilde{\al S}^j$ at $x_0$. Thus, the sheaves $\tilde{\al S}^0$, $\tilde{\al S}^1$ $\ldots$ will define a fine torsionless resolution of the constant sheaf. That is,
 $$0 \rightarrow \bb R \rightarrow \tilde{\al S}^0 \rightarrow \tilde{\al S}^1 \rightarrow \ldots$$
 is a fine torsionless resolution of the constact sheaf. 
 
 By the classical arguments of sheaf theory (see Warner \cite{Warner}), this implies that the resulting complex of global sections
 $$0 \rightarrow \tilde{\al S}^0(\tilde{X}) \rightarrow \tilde{\al S}^1(\tilde{X}) \rightarrow \ldots$$
 is isomorphic to the complex of flat currents on $X$. Therefore, the singular cohomology of $\tilde{X}$ is isomorphic to the flat cohomology of $X$. Thus, by Theorem \ref{GMnorm} it is clear that the flat cohomology of $X$ is isomorphic to the intersection cohomology in the top perversity. 
 \end{proof}

\subsection*{Acknowledgements} The first author was supported by FAPESP grant 2015/12667-5.


\end{document}